\documentclass[11pt]{amsart}

\usepackage{todonotes,marginnote}  
\usepackage{amsmath,amscd,amssymb,stmaryrd}
\usepackage{enumerate}
\usepackage{url}
\usepackage[colorlinks=true]{hyperref}

\def\Vol{\mathrm{Vol}}

\def\tr{\mathrm{tr}}

\def\pr{\mathbb{P}}

\def\ddbar{\sqrt{-1}\partial\bar{\partial}}

\def\Id{\mathrm{Id}}

\newtheorem{proposition}{Proposition}[section]
\newtheorem{lemma}[proposition]{Lemma}
\newtheorem{theorem}[proposition]{Theorem}
\newtheorem{newTheorem}{Theorem}
\newtheorem{corollary}[proposition]{Corollary}
\newtheorem{newcorollary}{Corollary}

\theoremstyle{definition}
\newtheorem{remark}[proposition]{Remark}
\newtheorem{definition}[proposition]{Definition}

\def\C{\mathrm{C}}

\def\Vol{\mathrm{Vol}}
\def\ddbar{\sqrt{-1}\partial\bar{\partial}}

\newcommand{\comp}{\mathbb{C}}

\newcommand{\scK}{\mathcal{K}}
\newcommand{\scL}{\mathcal{L}}

\newcommand{\scI}{\mathcal{I}}

\newcommand{\scX}{\mathcal{X}}
\newcommand{\scO}{\mathcal{O}}

\newcommand{\scB}{\mathcal{B}}
\newcommand{\scY}{\mathcal{Y}}

\DeclareMathOperator{\Bl}{Bl}

\DeclareMathOperator{\Grass}{Grass}

\DeclareMathOperator{\DF}{DF}

\usepackage{yfonts}  

\author[Y. Hashimoto and J. Keller]{Yoshinori Hashimoto and Julien Keller}

\dedicatory{In honor of Ngaiming Mok's 60th birthday}

\title[About J-flow, J-balanced metrics, uniform J and K-stability]{About J-flow, J-balanced metrics, uniform J-stability and K-stability}

\begin{document}
\bibliographystyle{alpha}

\begin{abstract}
From the work of Dervan-Keller \cite{DK}, there exists a quantization of the critical equation for the J-flow. This leads to the notion of J-balanced metrics. 
We prove that the existence of J-balanced metrics has a purely algebro-geometric characterization in terms of Chow stability, complementing the result of Dervan-Keller.  We also obtain various criteria that imply uniform J-stability and uniform K-stability, strengthening the results of Dervan-Keller. Eventually, we discuss the case of K\"ahler classes that may not be integral over a compact manifold.
\end{abstract}

\maketitle

\tableofcontents

\section{Introduction \label{intro}}

Let $M$ be a smooth projective manifold of complex dimension  $n\geq 2$. Given two K\"ahler forms $\omega$ and $\chi$, a priori in different classes, S.K. Donaldson introduced a flow of K\"ahler metrics called the J-flow in \cite{D9}. This flow, which has long time existence, is given by the following parabolic PDE in the $\omega$-potentials $\phi_t$: 
\begin{equation}\label{Jflow}
\frac{\partial \phi_t}{\partial t}= \gamma- \frac{\chi \wedge (\omega+\ddbar \phi_t)^{n-1}}{(\omega+\ddbar\phi_t)^n},
\end{equation} 
where $\gamma$ is a topological constant that can be computed by integration. Moreover, a \textit{critical metric} for Donaldson's J-flow is a solution at time $+\infty$ of the flow, i.e. a K\"ahler metric $\omega+\ddbar \phi$ solution of the elliptic equation
\begin{equation}\chi \wedge (\omega+\ddbar \phi)^{n-1} =\gamma (\omega+ \ddbar \phi)^n.\label{critical}
\end{equation}

A solution to the critical equation \eqref{critical} may not exist (see e.g.~\cite[Section 4.3]{D9}). There are several motivations for studying this flow and the existence of critical metrics. \\
Originally, Donaldson discovered that the critical equation corresponds to a geometric problem, namely to find a zero of a certain infinite dimensional moment map. \\
There is also a motivation in relationship with detecting the existence of constant scalar curvature K\"ahler metrics (CSCK in short). A key conjecture of T. Mabuchi and G. Tian asserts that the properness of the Mabuchi K-energy functional for the class $[\omega]$ should be equivalent to the existence of a CSCK metric in the class $[\omega]$, see \cite{Ma1,Ti2} (a small modification of the conjecture is
needed if the manifold admits holomorphic vector fields). Now, X.X Chen \cite{C2} (together with a result of J.Song and B. Weinkove) noticed that if $[\chi]$ is the canonical class, the existence of a critical metric in $[\omega]$ implies the properness of the Mabuchi functional for $[\omega]$. This relies on a decomposition formula of the Mabuchi functional. Therefore, if Mabuchi-Tian conjecture holds and $K_M$ is ample, the existence of a critical metric in $[\omega]$ should imply the existence of a CSCK metric in the same class. \\
Another motivation for studying the J-flow in connection to the CSCK metrics is the continuity method proposed by X.X.~Chen \cite{C3}, since it is known that the critical metric for the J-flow serves as the starting point of Chen's continuity method (cf.~\cite{Zeng} for the case when $\omega = \chi$ and \cite{YHtwisted} for the general case).\\

Consequently, the J-flow has attracted a great deal of attention in order to determine under which conditions it converges. For instance, in 2000, X.X. Chen \cite{C2} proved that if holomorphic bisectional curvature of $\chi$ is non-negative, there is convergence of the J-flow. But, thanks to the uniformization theorem of N. Mok \cite{Mok1}, these  manifolds are well-known and the decomposition formula of the Mabuchi functional does not hold, so we don't get any new information about its properness.\\
Later, several papers have provided criteria for existence of critical metrics, see \cite{C1,We2,We1,SW,FLSW}. From the point of view of PDE and geometric analysis, a necessary and sufficient condition for their existence was found in terms of positivity of certain differential forms by J. Song and B. Weinkove in \cite{SW}; see also the refinement \cite{Li-Shi-Yao2}. However, this analytic criteria is difficult to test in practice and it is expected that one can detect the existence of critical metrics by considering only an algebro-geometric condition. In that direction, M. Lejmi and G. Sz\'ekelyhidi \cite{L-Sz} provided a candidate for the algebro-geometric condition; later it was proved that their guess is correct in dimension 2 and for toric manifolds by T. Collins and G. Sz\'ekelyhidi \cite{Sz-Co}.  Nevertheless, in higher dimension it is very likely that the algebro-geometric condition of Lejmi-Sz\'ekelyhidi will be also hard to check in practice. It is natural to wonder if there are simple numerical criterion (in terms of Chern inequalities) for existence of critical metrics (cf.~\cite[Conjecture 1]{L-Sz}). Morally, over a general type manifold, this  would provide ``a neighborhood'' of the canonical class where there are classes endowed with CSCK metrics. The works \cite{We2,We1,SW,Panov-Ross,FLSW,SW2} have provided such conditions, probably close to be sharp for surfaces of general type,  but it turns out from \cite{DK} that these conditions are probably not optimal in higher dimension. In \cite{DK}, instead of searching a criterion for properness, it is obtained results in the direction of K-stability in the perspective of the Yau-Tian-Donaldson conjecture. The Yau-Tian-Donaldson conjecture asserts that the polarized manifold $(M,L)$ is K-stable if and only if there exists a CSCK metric in the class $c_1(L)$. Moreover \cite{DK} present provide a quantization of the critical equation \eqref{critical}, giving a canonical algebraic sequence of metrics that converges towards the critical metric. These algebraic metrics are called \textit{J-balanced}.\\

In this note, we present some results that complement the work of \cite{DK}. We prove that the existence of a J-balanced metric is equivalent to a GIT notion of stability in terms of the Chow scheme (Theorem \ref{balancedimpliesstable}). We also improve the results of \cite{DK} by considering \textit{uniform K-stability} instead of the classical notion of K-stability (Section \ref{jflowsection}). Actually, there are some reasons to expect that K-stability should be strengthened in order to obtain the Yau-Tian-Donaldson conjecture (see \cite{ACGT2,Sz2015}). A link between uniform K-stability and properness of the Mabuchi energy has been found explicitly in the non-Archimedean approach by S. Boucksom, T. Hisamoto and M. Jonsson \cite{BHJ}. Also, our motivation  is justified by the fact that existence of CSCK metric implies uniform K-stability, see \cite{BDL16}.\\ Eventually, the last part of the paper is dedicated to the study of uniform K-stability for non integral K\"ahler class over a compact K\"ahler manifold (Section \ref{nonintegralsect}).

\bigskip

{\small
\noindent {\bf Acknowledgments.} {\small Both authors would like to thank R. Dervan for very helpful conversations. They are also particularly grateful to the anonymous referee. \\
The work of both authors has been carried out in the framework of the Labex Archim\`ede (ANR-11-LABX-0033) and of the A*MIDEX project (ANR-11-IDEX-0001-02), funded by the ``Investissements d'Avenir" French Government programme managed by the French National Research Agency (ANR). The second author was also partially supported by supported by the ANR project EMARKS, decision No ANR-14-CE25-0010. 
}}

\section{J-balanced metrics \label{sect11}}

In this section we overview some results of \cite{DK} and introduce the notion of J-balanced metrics. This section provides a geometric quantization of the critical equation in the projective case. Consider now two ample line bundles $L_1,L_2$ on $M$. Fix hermitian metrics $h_1\in {\rm Met}(L_1)$, $h_2\in {\rm Met}(L_2)$ such that the curvatures $\omega = c_1(h_1)$, and $\chi = c_1(h_2)$ are both K\"ahler forms.  One can quantise metrics on $L_1$ using metrics on the finite dimensional vector spaces $H^0(M,L_1^k)$ with quantum parameter $k$, called \textit{Bergman metrics}. In \cite{DK}, it is introduced a flow on the space of Bergman metrics, which we call the J-balancing flow. Critical points of the J-balancing flow are called J-balanced metrics, and these fit naturally into a finite dimensional moment map picture. We provide now the details of the results.

\bigskip

Let us denote in the sequel $N=N_k=\dim H^0(L_1^k)-1$. We introduce a moment map setting in finite dimensions. Let us consider first 
$\mu_{FS} : \mathbb{P}^{N} \rightarrow \sqrt{-1}Lie(U(N+1))$ which is a moment map for the $U(N+1)$ action and the Fubini-Study metric $\omega_{FS}$ on $\mathbb{P}^{N}$. Given homogeneous unitary coordinates, one sets explicitly  $\mu_{FS}=(\mu_{FS})_{\alpha,\beta}$ as
\begin{equation}
\left(\mu_{FS}([z_0,...,z_N])\right)_{\alpha,\beta}=\frac{z_\alpha \bar{z}_\beta}{\sum_i |z_i|^2}. \label{mu}
\end{equation}
Then,  given an holomorphic embedding $\iota:M\hookrightarrow  \mathbb{P}H^0(L_1^k)^*$, and the Fubini-Study form $\omega_{FS}$ on the projective space, define
\begin{equation}\mu_{k,\chi}(\iota)=\frac{1}{\gamma}\int_M \mu_{FS}(\iota(p)) \;{\chi\wedge \iota^*(\omega_{FS}^{n-1})} (p).\label{Jmomentmap} \end{equation}
The map $\mu_{k,\chi}$ is an integral over $M$ of the  moment map for the $U(N+1)$ action over the space of all bases of $H^0(L_1^k)$, see \cite{DK}.
Note that if one defines a hermitian metric $H$ on $H^0(L_1^k)$, one can consider an orthonormal basis with respect to $H$ and the associated embedding, and thus (by abuse of notation) it also makes sense to speak of $\mu_{k,\chi}(H)$. In the {Bergman space} of metrics $\mathcal{B}=GL(N+1)/U(N+1),$ we have a preferred metric associated to an embedding $\iota : M \hookrightarrow \mathbb{P}H^0 (L_1^k)^*$ and this is precisely a J-balanced metric.

\begin{definition}[J-balanced embedding, J-balanced metric]\label{mu0kchi}
The embedding $\iota$ is J-balanced if and only if $$\mu^0_{k,\chi}(\iota):=\mu_{k,\chi}(\iota) - \frac{\tr(\mu_{k,\chi}(\iota))}{N+1}\Id_{N+1}=0.$$
\end{definition}
\noindent A J-balanced  embedding corresponds (up to $SU(N+1)$-isometries) to a J-balanced metric $\iota^*\omega_{FS}$ by pull-back of the Fubini-Study metric from $\mathbb{P}H^0(L^k_1)^*$. Note that for $H\in {\rm Met}(H^0(L_1^k))$,  it also makes sense to consider $\mu_{k,\chi}(h)$ where $h=FS(H)\in {\rm Met}(L_1^k)$, i.e when $h$ belongs to the space of  \textit{Bergman} type fibrewise metric that we identify with $\mathcal{B}$.

\medskip 

\par On the other hand, seen as a hermitian matrix, $\mu^0_{k,\chi}(\iota)$ induces a vector field on $\mathbb{P}^{N}$. The balancing flow is defined by
\begin{equation*}
\frac{d \iota(t)}{dt} =  -  \mu^0_{k,\chi}(\iota(t)),  
\end{equation*}
To fix the starting point of this flow, we choose a K\"ahler metric $\omega=\omega(0)$ and we construct a sequence of hermitian metrics $h_k(0)$ such that $\omega_k(0):=c_1(h_k(0))$ converges smoothly to $\omega(0)$ providing a sequence of embeddings $\iota_k(0)$ for $k\gg 0$. For technical reasons, we decide to rescale this flow by 
considering the following ODE,
\begin{equation}
\frac{d \iota_k(t)}{dt} =  - k \mu^0_{k,\chi}(\iota_k(t)),  \label{resbalflowJ}
\end{equation}
which we call the \textit{rescaled J-balancing flow}. This flow induces a sequence of K\"ahler metrics 
$$\omega_{k}(t)=\frac{1}{k}\iota_k(t)^*(\omega_{FS}),$$ when $t$ and $k$ tends to infinity. The behavior of the sequence $\omega_{k}(t)$ is fully described in the next result.

\begin{theorem}[\cite{DK}]\label{thm1} Fix $T>0$ and let $\omega_k(t)$ be the solution of the J-balancing flow, for $t\in  [0,T]$. Then as $k\to\infty$, the sequence $\omega_k(t)$ converges in $C^{\infty}$ to the solution of the J-flow as $k\to\infty$. Furthermore, the convergence is $\C^1$ in the variable $t$. \\
Assuming there is a critical point of the J-flow, the convergence holds for all $t>0$. \end{theorem}
Building on earlier works of Y. Sano \cite{Sa} and R. Seyyedali \cite{Sey2}, this approach also provides a iterative scheme to compute numerically J-balanced metrics.
As J-balanced metrics are unique, one can recover the uniqueness of the critical metrics as was proved in \cite{C1}. We have also the following consequence of the previous result and of the long time existence of the J-flow.

\begin{corollary}[\cite{DK}]\label{cor1}
Consider $(M,L_1,L_2)$ a polarized manifold by $L_1,L_2$ such that that there exists a critical metric
solution of \eqref{critical}. Then for $k$ sufficiently large, there exists a sequence of J-balanced metrics on 
${\rm Met}(L_1^k)$ obtained as the limit of the balancing flow at time $t=+\infty$. Furthermore, the sequence of J-balanced metrics converges in smooth topology towards the critical metric when $k\rightarrow +\infty$. 
\end{corollary}

\section{Relationship with Chow stability}\label{chow-stabl}

\subsection{Chow stability for a linear system} \label{chow-stabl-lins}

In this section we recall from \cite{DK} the GIT notion of \emph{Chow stability for a linear system} $|L_2|$ in a fixed polarized variety $(M,L_1)$. 
 
Let $V$ be a vector space (which is later taken to be $H^0 (M,L_1^r)^*$) and $Y\subset\pr(V) \cong \pr^N$ be a fixed subvariety. We denote by $m$ the dimension of $Y\subset\pr^N$ and $d$ the degree of $Y$. Let $Z$ be the set of $(N-m-1)$-dimensional planes intersecting $Y$ nontrivially, so that $Z\subset \Grass(N-m,N+1)$. We denote the Pl\"ucker embedding of this Grassmannian as \begin{equation}\label{plucker}Pl: \Grass(N-m,N+1) \hookrightarrow \pr(\Lambda^{N-m}V).\end{equation} As $Z$ has codimension one, it is given as the vanishing set of some holomorphic section $f\in H^0(\Grass(N-m,N+1), \scO(d))$ unique up to scaling. One therefore has a corresponding point $[f]\in \pr(H^0(\Grass(N-m,N+1), \scO(d)))$, called the \emph{Chow point}. The $SL(N+1,\comp)$ action on $\pr^N$ induces one on $\pr(H^0(\Grass(N-m,N+1), \scO(d)))$. Then, $Y\subset\pr^N$ is \emph{Chow stable} if its Chow point $[f]$ is GIT stable under the induced action of $SL(N+1,\comp)$. The Hilbert-Mumford criterion states it is enough to show a corresponding weight is positive for each 1-parameter subgroup $\lambda:\comp^*\hookrightarrow SL(N+1,\comp)$. Fixing some 1-parameter subgroup $\lambda$, the limit cycle $Y_0=\lim_{t\to 0}\lambda(t).Y$ is naturally a point in the same Chow variety as $Y$, which has a corresponding Hilbert-Mumford weight. Hence $Y$ is Chow stable if and only if these weights for $Y_0$ are strictly positive for each 1-parameter subgroup.

 This leads to consider the following definition.
 \begin{definition}[\cite{DK} Twisted Chow stability] \label{defchowstable}Let $Y$ be a subvariety of $M$, where $M$ is polarized by $L_1$. We say that $Y$ is $M$-\emph{twisted Chow stable at the level $r$} if it is Chow stable under the embeddings $Y\hookrightarrow\pr(H^0(M,L_1^r)^*)$, i.e.~if the Chow point associated to the embedding $Y\hookrightarrow\pr(H^0(M,L_1^2)^*)$ is GIT stable under the action of $SL(N+1)$; we also say that $M$-\emph{twisted asymptotic Chow stable} if it is $M$-twisted Chow stable at level $r$ for all $r\gg 0$.\end{definition}

We note that twisted Chow stability  is a bona fide, finite-dimensional, GIT stability notion as long as we fix the exponent $r$. This implies, in particular, that the Hilbert-Mumford weight of the 1-parameter subgroup $\mu(\lambda,Y)$ can be computed from the action on $\pr(H^0(M,L_1^r)^*)$ itself. For that, we can consider a geometrization of the 1-parameter subgroups considered above, which eventually leads to an alternative definition of the twisted Chow stability in Theorem \ref{numericalstability}.

\begin{definition} A \emph{test configuration} $(\scX,\scL)$ for a polarized variety $(M,L_1)$ is a variety $\scX$ together with
\begin{itemize} 
\item a proper flat morphism $\pi: \scX \to \comp$,
\item a $\comp^*$-action on $\scX$ covering the natural action on $\comp$,
\item and an equivariant relatively very ample line bundle $\scL$ on $\scX$
\end{itemize}
such that the fibre $(\scX_t,\scL_t)$ over $t\in\comp$ is isomorphic to $(M,L_1^r)$ for one, and hence all, $t \in \comp^*$ and for some $r>0$. We call $r$ the \emph{exponent} of the test configuration. A test configuration $(\scX, \scL)$ is called \emph{trivial} if $\scX \cong M \times \comp$ with trivial $\comp^*$-action on $M$. \end{definition}

\begin{remark} \label{remrossthomas}
By Proposition 3.7 of \cite{RT2}, a test configuration $(\scX,\scL)$ of exponent $r$ can always be obtained as a closure of a $GL$-orbit in $\pr (H^0(M,L^r)^*)$, i.e. we may assume that for any $(\scX,\scL)$ there exists $A \in \mathfrak{gl}(H^0 (M,L^r))$ with rational eigenvalues such that $\scX$ is equal to the flat closure of the $GL$-orbit $\{ e^{tA} \cdot M\} \subset \pr (H^0(M,L^r)^*)$. It is well known that we may assume $A$ to be hermitian \cite{D5}. Note that $(\scX_A, \scL_A)$ is a trivial test configuration if and only if $A$ is a multiple of the identity. Later we shall write $(\scX_A, \scL_A)$ for the test configuration obtained this way by $A \in \mathfrak{gl}(H^0 (M,L^r))$.
\end{remark}

\begin{remark} \label{defflatlimit}
The central fibre $\mathcal{X}_{A,0}$ of the above test configuration is the flat limit of $M$ under the 1-parameter subgroup $\{e^{tA}\}$, which can be defined by taking the ``limit of the defining equations''. If $I$ is the homogeneous ideal defining $M \subset \pr(H^0(M,L^r)^*)$, we can define the initial term of $f \in I$ to be the term of the decomposition of $f$, denoted $\mathrm{in}(f)$, for which $\{e^{tA}\}$ acts with the smallest weight. Then, $\mathcal{X}_{A,0}$ is defined by the ideal $I_0=\{\mathrm{in}(f),f\in I\}$ generated by the set of initial terms of elements in $I$. The reader is referred to Sz\'ekelyhidi's textbook \cite[Sections 6.2 and 6.3]{szebook} for more discussions on the flat limit. 
\end{remark}

Let $(\scX,\scL)$ be a test configuration. As the $\comp^*$-action on $(\scX,\scL)$ fixes the central fibre, there is an induced $\comp^*$-action on $(\scX_0,\scL_0)$ and hence on $H^0(\scX_0,\scL^K_0)$ for each $K$. We denote the Hilbert polynomial and total weight of this action respectively by \begin{align*} h(K)&=a_0K^n+a_1K^{n-1}+O(K^{n-2}), \\ w(K) &=b_0K^{n+1}+b_1K^n+O(K^{n-1}).\end{align*} By asymptotic Riemann-Roch and flatness of the test configuration, we have intersection-theoretic formulas for $a_0,a_1,$ as  $$a_0 = r^n\frac{L^n}{n!}, a_1 = r^{n-1}\frac{K_M.L^{n-1}}{2(n-1)!}.$$ 
	
	Observe now that the test configuration $(\scX,\scL)$ for $(M,L_1)$ naturally induces the one $(\scY,\scL |_{\scY})$ for $(Y, L_1 |_Y)$. For its central fibre $(\scY_0,\scL_0)$, denote the corresponding Hilbert and weight polynomials by \begin{align*} \hat{h}(K)&=\hat{a}_0K^m+O(K^{m-1}), \\ \hat{w}(K) &=\hat{b}_0K^{m+1}+O(K^{m}),\end{align*} where $m$ is the dimension of $Y$.

\begin{theorem}[\cite{DK}]\label{numericalstability} Let $(M,L_1)$ be a polarized variety and $Y\subset M$ a subvariety. Consider the 
polynomial expression $$\hat{w}_{r,k}=\hat{w}(k)rh(r)-kw(r)\hat{h}(k)>0$$ in the variables $k$ (setting $k=Kr$) and $r$. As $\hat{w}_{r,k}$ has degree $(m+1)$ in $k$, it can be written as
$$\hat{w}_{r,k} = \sum_{i=0}^{m+1}e_i(r)k^i.$$
Then $Y$ is $M$-twisted Chow stable at the level $r$ if and only if $e_{m+1} (r) >0 $. Moreover,  $Y$ is $M$-twisted asymptotic Chow stable  if and only if 
for all $r\gg 0$, for all test-configurations of exponent $r$, we have $e_{m+1}(r)>0$.
\end{theorem}

\begin{remark} \label{remchowwte}
Observe that we can write $e_{m+1} (r) = \hat{b}_0 r h(r) - w(r) \hat{a}_0$.
\end{remark}

Focussing now on the case $Y$ is a divisor of $M$, we can define the Chow weight of a linear system thanks to next lemma (cf.~\cite[Lemma 4.1]{DK}).

\begin{lemma}[\cite{DK}] \label{wtcstzariski} Let $M\subset\pr^N$ be a projective variety together with a linear system $|L_2|$. For each 1-parameter subgroup $\lambda\hookrightarrow SL(N+1,\comp)$, the Chow weight of $\lambda$ for $D\in |L_2|$ is constant outside a Zariski closed subset of $|L_2|$. We define the Chow weight of $\lambda$ for $|L_2|$ to equal this general value. \end{lemma}

The above lemma holds even when we take $\lambda$ to be a limit of 1-parameter subgroups so that the image in $SL(N+1,\mathbb{C})$ has irrational eigenvalues. 
This  is clear from the proof of \cite[Lemma 4.1]{DK}.

\begin{definition}[\cite{DK} Chow stability of a linear system]\label{defasymChow} Let $M\subset\pr^N=\mathbb{P}(H^{0}(M,L_1^r)^*)$ be a projective variety, polarized by $L_1$, and let $|L_2|$ be a linear system.
We say that a linear system $|L_2|$ is \emph{Chow stable at level $r$} if the Chow weight of each (nontrivial) test configuration of exponent $r$ is strictly positive. This is equivalent to the fact that for each 1-parameter subgroup $\lambda\hookrightarrow SL(h^0(M,L_1^r),\comp)$, the Chow weight of $\lambda$ for $|L_2|$ is strictly positive.
We say that $|L_2|$ is \emph{asymptotically Chow stable} if it is Chow stable at level $r$ for all $r\gg 0$.
\end{definition}

\begin{remark}
 Unlike the twisted Chow stability notion for a fixed subvariety, the Chow stability of a linear system is not a bona fide GIT notion; although the Chow weight of a 1-parameter subgroup $\lambda$ is constant outside of a Zariski closed subset, this Zariski closed subset to be removed depends on $\lambda$. Nevertheless, we shall see that Kempf-Ness theorem still holds true as we are discussing in next section.
\end{remark}

\subsection{J-balanced metrics and asymptotic Chow stability}

In this section we work with a smooth projective variety $M$ embedded in a fixed projective space $\pr^N$, and with a linear system $|L_2|$. We prove the following result that strengthens a result of \cite{DK} and provides a complete algebro-geometric description of J-balanced metrics.

\begin{newTheorem}\label{balancedimpliesstable} Consider a polarized complex manifold $(M,L_1)$ and an auxiliary ample line bundle $L_2$ on $M$. Then $(M,L_1,L_2)$ admits a J-balanced metric in the K\"ahler class $c_1 (L_1^r)$ if and only if the linear system $|L_2|$ is Chow stable for $M \subset \mathbb{P} (H^0 (M,L_1^r)^*)$. \end{newTheorem}

In \cite{DK}, it is proved one sense of this equivalence, namely that the existence of a critical metric of the J-flow implies that the linear system $|L_2|$ is asymptotically Chow stable. Thus it suffices to prove that the Chow stability of $|L_2|$ implies the existence of a J-balanced metric, which is the aim of the present section.

Before beginning the proof of this theorem, we need to recall some functionals defined in \cite{DK}. The reader is also referred to \cite{P-S4} for the review of background materials.
\begin{definition} \label{definitionaymfn}Consider $D\in |L_2|$ a smooth divisor in $M$. We define the functional $I_D^{\text{AYM}}$ over the space of $Met(L_1)$ variationally by $$\frac{d}{dt}I^{\text{AYM}}_D(\phi(t)) = -\frac{1}{\Vol_{{L_1}}(D)} \int_D \dot{\phi}_t \omega_{\phi_t}^{n-1},$$ taking the value zero at $\phi=0$. \\
Here $\omega_{\phi_t}:=\omega+\ddbar\phi_t$ and $\Vol_{{L_1}}(D)=\int_D c_1(L_1)^{n-1} = \int_M c_1(L_1)^{n-1}.c_1(L_2)$ is the volume of $D$. \end{definition} 

We then define a functional $J_{\chi}$ by
$$J_{\chi}(\phi) = -\Vol_{L_1}(M)\int_{D\in|L_2|} I_D^{\text{AYM}}(\phi)d\mu,$$
where the signed measure $d\mu$ is defined as follows.

\begin{theorem}\cite{L-Sz}\label{generalelement} Let $M$ be a smooth projective $n$-dimensional variety together with a very ample line bundle $L_2$. Let $\alpha\in c_1(L_2)$ be a positive $(1,1)$-form. Then there is a smooth signed measure $\mu$ on the projective space $|L_2|$ such that $$\alpha = \int_{D\in|L_2|}[D]d\mu$$ holds in the weak sense, i.e. for all smooth $(n-1,n-1)$-forms $\beta$ we have $$\int_M \alpha\wedge\beta = \int_{D\in|L_2|}\left(\int_D \beta\right)d\mu.$$ \end{theorem}
Here we may integrate only over the smooth elements $D\in|L_2|$, since the complement has measure zero (by Bertini's Theorem) this does not affect the value of the integral. 

Writing $\mathcal{B}_r$ for the space of positive definite hermitian matrices on the vector space $H^0(M,L_1^r)$, recall (cf. \cite[Section 3.1]{DK}) that there exists a functional $I_{\mu^0_{r, \chi}} : \mathcal{B}_r \to \mathbb{R}$ defined as
	\begin{equation}
		I_{\mu^0_{r, \chi}}(H):= J_{\chi}\circ FS (H) + \frac{Vol_{L_1}(M)}{h^0(M,L_1^r)} \log \det (H). \label{jbalfunc}
\end{equation}
This functional is the {\it integral of the moment map} $\mu^0_{r, \chi}$ defined in Definition \ref{mu0kchi}, in the sense that its differential gives $\mu^0_{r, \chi}$. For the notion of integral of a moment map, we refer to \cite[Section 3]{MR}. Such functional has the property to have its critical points that coincide with zeros of the moment map in the complex orbit and enjoys properties of convexity and cocyclicity.
Consequently, we have the following result.
\begin{theorem} \cite[Corollary 3.5]{DK} \label{varcharjbal}
	$I_{\mu^0_{r, \chi}}$ is a geodesically convex functional on $\mathcal{B}_r$ whose unique critical point is the J-balanced metric.
\end{theorem}

Theorem \ref{varcharjbal} means that we only need to show that $I_{\mu^0_{r, \chi}}$ has a critical point, in order to prove that $(M,L_1,L_2)$ admits a J-balanced metric. Since $I_{\mu^0_{r, \chi}} $ is geodesically convex and $\mathcal{B}_r$ is geodesically complete (with respect to the bi-invariant metric), we only need to show that for any geodesic $\left\{ H(t) \right\} \subset \mathcal{B}_r$ we have
\begin{equation}
	\lim_{t \to \infty} \frac{d}{dt} I_{\mu^0_{r, \chi}} (H(t)) > 0 \label{jbalfuncslinf}
\end{equation}
in order to prove that $I_{\mu^0_{r, \chi}}$ has a critical point.

Thus, in order to prove Theorem \ref{balancedimpliesstable}, it suffices to show that Chow stability of $|L_2|$ implies (\ref{jbalfuncslinf}) for any geodesic $\left\{ H(t) \right\} \subset \mathcal{B}_r$.

As in Remark \ref{remrossthomas}, we consider test configurations $(\scX_A, \scL_A)$ of exponent $r$ generated by a hermitian matrix $A \in \mathfrak{gl} (H^0(M,L_1^r))$ with rational eigenvalues. Note that, by Remark \ref{remchowwte}, Chow stability of $|L_2|$ implies that for all $D \in |L_2|$ outside of a Zariski closed subset we have
\begin{equation*}
	r \hat{b}_0 - \frac{\hat{a}_0}{h(r)}w(r) >0,
\end{equation*}
for all choices of hermitian matrices $A $ with rational eigenvalues giving rise to $(\scX_A, \scL_A)$. Recalling Lemma \ref{wtcstzariski}, the above weight is constant for all $D \in |L_2|$ outside of a Zariski closed subset. We thus get
\begin{equation}
	\int_{D \in |L_2|}\left(   r \hat{b}_0 - \frac{\hat{a}_0}{h(r)}w(r)  \right) d\mu >0 \label{ineqchow}
\end{equation}
for all hermitian matrices $A$ with rational eigenvalues.

Our aim is to relate the above inequality to (\ref{jbalfuncslinf}). The key result that we need to establish this connection is the following theorem of Donaldson.

\begin{theorem}[\emph{\cite[Proposition 3]{D5}}]
	Let $\left\{ H(t) \right\} \subset \mathcal{B}_r$ be a geodesic defined by $H(t) = e^{-At} (e^{-At})^*$ with $A$ having rational eigenvalues. Then we have
	\begin{equation} \label{D5Prop3}
		\lim_{t \to \infty} \frac{1}{2} \frac{d}{dt} I_D^{\mathrm{AYM}} (FS(H(t))) = - r  \frac{\hat{b}_0}{\hat{a}_0}
	\end{equation}
	where $\hat{b}_0$ is defined in terms of the test configuration $(\scX_A, \scL_A)$ of exponent $r$.
\end{theorem}
\begin{remark}
 There is an obvious typo in the statement of \cite[Proposition 3]{D5} where $b_1k^{n+1}$ is actually the leading term $b_0k^{n+1}$ (our convention is such that the K\"ahler metrics are in $c_1(L_1)$ and not $r c_1(L_1)$ which leads to a different normalization to  \cite[Proposition 3]{D5}). \\
Note also that the factor $\hat{a}_0$ does not depend on $A$.
\end{remark}

Later we need to extend the above theorem to hermitian matrices with potentially non-rational eigenvalues. We provide the following lemma, so that we can apply rational approximation argument when we deal with non-rational eigenvalues.

\begin{lemma}  \label{continuityb0} 
For any hermitian matrix $A$ (with possibly non-rational eigenvalues), there exists a sequence $\left\{ A_p \right\}_p$ of hermitian matrices with rational eigenvalues such that $A_p \to A$ as $p \to \infty$ and
\begin{equation*}
\lim_{p \to \infty} \lim_{t \to \infty} \frac{1}{2} \frac{d}{dt} I_D^{\mathrm{AYM}} (FS(A_p)) = \lim_{t \to \infty} \frac{1}{2} \frac{d}{dt} I_D^{\mathrm{AYM}} (FS(\lim_{p \to \infty} A_p)),
\end{equation*}
where we wrote $FS(A_p)$ for $FS(e^{-A_p t} (e^{-A_p t})^*)$. In particular, writing $\hat{b}_{0,p}$ for the $\hat{b}_0$ defined with respect to $A_p$, we see that there exists a real number $\tilde{b} := \lim_{i \to \infty} \hat{b}_{0,p}$ such that
\begin{equation*}
	\lim_{t \to \infty} \frac{1}{2} \frac{d}{dt} I_D^{\mathrm{AYM}} (FS(A)) = - r  \frac{\tilde{b}}{\hat{a}_0}.
\end{equation*}
\end{lemma}

\begin{proof}
	From Definition \ref{definitionaymfn} and writing down the Fubini-Study metric explicitly, we observe that
	\begin{equation} \label{inthamiltonianx0}
	\lim_{t \to \infty} \frac{1}{2} \frac{d}{dt} I_D^{\mathrm{AYM}} (FS(A)) = \int_{|\scY_0|} \frac{A_{ij} Z_i \bar{Z}_j}{\sum_l |Z_l|^2} \frac{\omega^{n-1}_{FS}}{(n-1)!},
	\end{equation}
	where
	\begin{itemize}
		\item $\{Z_i \}_i$ denotes homogeneous coordinates on $\mathbb{P} (H^0 (M,L^r)^*)$,
		\item $\omega_{FS}$ is a Fubini-Study metric on the ambient space $\mathbb{P} (H^0 (M,L^r)^*)$, corresponding to the choice of an orthonormal basis $\{ Z_i \}_i$,
		\item $|\scY_0|$ denotes the limit cycle $\lim_{t \to \infty} e^{At} \cdot D \subset \mathbb{P} (H^0 (M,L^r)^*)$.
	\end{itemize}
	Now we construct an approximating sequence $\{ A_p \}_p$ by matrices with rational eigenvalues, such that the central fibres defined by $A_p$'s are all isomorphic to the one defined by $A$. Recalling Remark \ref{defflatlimit}, it suffices to show that the flat limit defined by $A_p$'s agrees with the one defined by $A$.

	We can choose the homogeneous coordinates $\{ Z_i \}_i$ appropriately so that $A = \text{diag}(\lambda_1 , \dots , \lambda_{h(r)})$, $\lambda_1 \ge \cdots \ge \lambda_{h(r)}$. Before we continue with the proof, we look at the following simple example to have an intuition about certain aspects of the proof. Let $(\lambda_1, \lambda_2, \lambda_3) = (\sqrt{2}, 0 ,-\sqrt{2})$ and $Z_2^2 - Z_1 Z_3 - Z_1 Z_2$ be an element in the ideal defining $M$. $A$ acts on this equation as $(Z_2^2 - Z_1 Z_3) - t^{\sqrt{2}}Z_1 Z_2$, so its initial term (cf.~Remark \ref{defflatlimit}) is $Z_2^2 - Z_1 Z_3$. Note that if there is another weight $(\mu_1, \mu_2, \mu_3)$ that has the same initial term as $(\lambda_1, \lambda_2, \lambda_3)$, it has to satisfy the $\mathbb{Z}$-linear relationship $2 \mu_2 = \mu_1 + \mu_3$; this is essentially because the initial term $Z_2^2- Z_1Z_3$ of $Z_2^2 - Z_1 Z_3 - Z_1 Z_2$ is not a monomial.

	Now resuming the proof, we pick the generators of the ideal $I_0$ of the flat limit defined by $A$ that are not monomials. We can write down all the $\mathbb{Z}$-linear relationships among $\lambda_i$'s, and call them $l_1, \dots, l_q$, say. Finding a sequence $\{ A_p \}_p$ of rational matrices $A_p = \text{diag}(a_{1,p}, \dots , a_{h(r),p})$ with the same flat limit is equivalent to finding a sequence of vectors $\{ (a_{1,p}, \dots , a_{h(r),p} ) \}_p$ in $\mathbb{R}^{h(r)}$ such that
	\begin{enumerate}
		\item $a_{i,p} \in \mathbb{Q}$ for all $p$ and $i$,
		\item $a_{1,p} \ge \cdots  \ge a_{h(r),p}$,
		\item $(a_{1,p}, \dots , a_{h(r),p})$ satisfies $l_1 , \dots, l_q$,
		\item $(a_{1,p}, \dots , a_{h(r),p})$ is sufficiently close to $(\lambda_1 , \dots , \lambda_{h(r)})$, so that the initial terms defined by them are the same.
	\end{enumerate}
	Now, writing $\Delta$ for the convex set in $\mathbb{R}^{h(r)}$ defined by the inequality $a_1 \ge \cdots \ge a_{h(r)}$, and $H$ for the linear subspace cut out by $l_1 , \dots , l_q$, the problem is to find a sequence of rational points in $\Delta \cap H$ converging to $(\lambda_1 , \dots , \lambda_{h(r)}) \in \Delta \cap H$. Since $l_1 , \dots , l_q$ have coefficients in $\mathbb{Z}$, there is a non-empty open subset of $H$ around $(\lambda_1 , \dots , \lambda_{h(r)})$ which contains infinitely many rational points that are also included in $\Delta$. Since $\mathbb{Q}$ is dense in $\mathbb{R}$, we can find the sequence $\{ (a_{1,p}, \dots , a_{h(r),p}) \}_p$ as above, as expected.

Thus we have found a sequence $\{ A_p \}_p$ converging to $A$, such that the central fibres of the test configurations $(\scY_{A_p}, \scL_{A_p})$ for $(D,L |_D)$ are all isomorphic, and agrees with the limit cycle $|\scY_0|$. This implies that (\ref{inthamiltonianx0}) is continuous in $A$, and establishes the claimed statement.
\end{proof}



We now come back to our original task of relating the inequalities (\ref{jbalfuncslinf}) and (\ref{ineqchow}). By the Lebesgue integration theorem, we have
\begin{align*}
	\lim_{t \to \infty} \frac{1}{2} \frac{d}{dt} J_{\chi} \circ FS(H(t)) &= -\frac{1}{2} \Vol_{L_1}(M)\int_{D\in|L_2|} \left( \lim_{t \to \infty} \frac{d}{dt}  I_D^{\text{AYM}}(\phi) \right) d\mu, \\
	&=r\frac{a_0}{\hat{a}_0} \int_{D \in |L_2|} \hat{b}_0 d\mu.
\end{align*}
Combined with $\log \det H(t) = -2 \tr(tA)$, we get
\begin{equation*}
	\frac{1}{2} \lim_{t \to \infty} \frac{d}{dt} I_{\mu^0_{r, \chi}} (H(t)) = r\frac{a_0}{\hat{a}_0}  \int_{D \in |L_2|} \hat{b}_0 d\mu - \frac{a_0}{h(r)} \tr(A).
\end{equation*}
Observing $\tr(A) $ is equal to the weight $ w(r)$ of the $\comp^*$-action defining $(\scX_A , \scL_A)$ and recalling that $w(r)$ is independent of $D \in |L_2|$ (since it is defined by $A$ acting on $H^0(M,L_1^r)$), we can write the above as follows.
\begin{lemma} \label{jbalchowrat}
	Given a test configuration $(\scX_A , \scL_A)$ generated by a hermitian matrix $A$ with rational eigenvalues, we have
\begin{equation*}
	\frac{1}{2} \lim_{t \to \infty} \frac{d}{dt} I_{\mu^0_{r, \chi}} (H(t)) = \frac{a_0}{\hat{a}_0}  \int_{D \in |L_2|} \left( r \hat{b}_0  - \frac{\hat{a}_0}{h(r)} w(r) \right) d\mu
\end{equation*}
for $H(t) = e^{-tA} (e^{-tA})^*$.
\end{lemma}
Thus, (\ref{ineqchow}) implies that $\lim_{t \to \infty} \frac{d}{dt} I_{\mu^0_{r, \chi}} (H(t)) >0$ holds for all geodesic paths $\left\{ H(t) \right\}$ in $\mathcal{B}_r$ of the form $e^{-tA} (e^{-tA})^*$ with $A$ having rational eigenvalues. We need to prove that $\lim_{t \to \infty} \frac{d}{dt} I_{\mu^0_{r, \chi}} (H(t)) >0$ holds for \textit{all} geodesic paths in $\mathcal{B}_r$, which we achieve in the following.

\begin{lemma} \label{jbalchowirrat}
	Given any hermitian matrix $A$, we have
\begin{equation*}
	\lim_{t \to \infty} \frac{d}{dt} I_{\mu^0_{r, \chi}} (H(t)) > 0
\end{equation*}
for $H(t) = e^{-tA} (e^{-tA})^*$.
\end{lemma}


\begin{proof}
	Since $\mathbb{Q}$ is dense in $\mathbb{R}$ it is obvious that Lemmas \ref{continuityb0} and \ref{jbalchowrat} imply
	\begin{equation} \label{nonnegjbal}
	\lim_{t \to \infty} \frac{d}{dt} I_{\mu^0_{r, \chi}} (H(t)) \ge 0
	\end{equation}
for $H(t) = e^{-tA} (e^{-tA})^*$ with any hermitian matrix $A$.



Recalling Lemma \ref{continuityb0}, we see that there exists a real number $\tilde{b}$ such that the equality as in Lemma \ref{jbalchowrat}
\begin{equation}
	\frac{1}{2} \lim_{t \to \infty} \frac{d}{dt} I_{\mu^0_{r, \chi}} (H(t)) = \frac{a_0}{\hat{a}_0}  \int_{D \in |L_2|} \left( r \tilde{b}  - \frac{\hat{a}_0}{h(r)} w(r) \right) d\mu \ge 0 \label{eqjbalchowgenirrat}
\end{equation}
still holds for any (not necessarily rational) hermitian matrices $A$. Note that the integrand $ r \tilde{b}  - \frac{\hat{a}_0}{h(r)} w(r) $ is the Chow weight of $D$ in $\vert L_2 \vert$ with respect to the 1-parameter subgroup defined by $A$; this can be seen by taking rational approximations of $A$ as in Lemma \ref{continuityb0}. This implies in particular that the value of the integrand is constant outside of a Zariski closed subset of the linear system $\vert L_2\vert$ by Lemma \ref{wtcstzariski}.

On the other hand, we can diagonalize $A$ as $A = \mathrm{diag}(\lambda_1 , \cdots, \lambda_{h(r)})$, $\lambda_1 \ge \cdots \ge \lambda_{h(r)}$ and write it as a sum of two hermitian matrices $A_{\alpha}$, $A_{\beta}$ satisfying the following properties
\begin{enumerate}
	\item $A_{\alpha} = \mathrm{diag} (\alpha_1 , \cdots , \alpha_{h(r)})$ with $\alpha_i \in \mathbb{Q}$ for all $i$ and $\alpha_1 \ge \cdots \ge \alpha_{h(r)}$,
	\item $A_{\beta} = \mathrm{diag} (\beta_1 , \cdots , \beta_{h(r)})$ with $\beta_i \in \mathbb{R}$ for all $i$ and $\beta_1 \ge \cdots \ge \beta_{h(r)}$,
	\item the central fibre of the test configurations defined by $A$, $A_{\alpha}$, $A_{\beta}$ are all isomorphic.
\end{enumerate}
The fact that we can take such $A_{\alpha}$ and $A_{\beta}$ follows from an argument as in the proof of Lemma \ref{continuityb0}; note that in the notation therein, we may further take $A_{\beta} = A - A_{\alpha}$ to lie in $\Delta \cap H$, by considering an $\epsilon$-ball around $A$ which contains infinitely many rational points.

The third property means that the weight $w(r)$ of $A$ (for the exponent $r$) can be written as $w(r) = w_{\alpha} (r) + w_{\beta} (r)$, with $w_{\alpha} (r)$ corresponding to $A_{\alpha}$ and $w_{\beta} (r)$ corresponding to $A_{\beta}$, since these weights are all with respect to the same central fibre.

Moreover, recalling (\ref{inthamiltonianx0}) and the notation used therein, we have
\begin{align}
&\lim_{t \to \infty} \frac{1}{2} \frac{d}{dt} I_D^{\mathrm{AYM}} (FS(A)) \notag \\
&= \lim_{t \to \infty} \frac{1}{2} \frac{d}{dt} I_D^{\mathrm{AYM}} (FS(A_{\alpha})) + \lim_{t \to \infty} \frac{1}{2} \frac{d}{dt} I_D^{\mathrm{AYM}} (FS(A_{\beta})), \label{keylinearity}
\end{align}
and hence
\begin{equation*}
	\tilde{b} = \hat{b}_{\alpha,0} + \tilde{b}_{\beta},
\end{equation*}
where $\hat{b}_{\alpha,0}$ is the leading term of the expansion $\hat{w}_{\alpha} (r) = \hat{b}_{\alpha,0} r^{n}  + \cdots$ and $\tilde{b}_{\beta}$ is defined by $A_{\beta}$ as in Lemma \ref{continuityb0}.
Then, (\ref{eqjbalchowgenirrat}) implies
\begin{align*}
	&\frac{1}{2} \lim_{t \to \infty} \frac{d}{dt} I_{\mu^0_{r, \chi}} (H(t)) \\
	&= \frac{a_0}{\hat{a}_0}  \int_{D \in |L_2|} \left( r \tilde{b}  - \frac{\hat{a}_0}{h(r)} w(r) \right) d\mu \\
	&=\frac{a_0}{\hat{a}_0}  \int_{D \in |L_2|} \left( r \hat{b}_{\alpha,0}  - \frac{\hat{a}_0}{h(r)} w_{\alpha}(r) \right) d\mu + \frac{a_0}{\hat{a}_0}  \int_{D \in |L_2|} \left( r \tilde{b}_{\beta}  - \frac{\hat{a}_0}{h(r)} w_{\beta}(r) \right) d\mu,
\end{align*}
by noting that the decomposition in the third line is possible since the integrands are constant outside of a Zariski closed subset, thanks to Lemma \ref{wtcstzariski}.
The first term in the third line is strictly positive by (\ref{ineqchow}), and the second term is non-negative by (\ref{nonnegjbal}), and hence $ \lim_{t \to \infty} \frac{d}{dt} I_{\mu^0_{r, \chi}} (H(t)) >0$ as desired.
\end{proof}

This finally achieves proving Inequality (\ref{jbalfuncslinf}), and completes the proof of Theorem \ref{balancedimpliesstable}. Note that the key to the proof is Equality (\ref{keylinearity}) which is based on the explicit formula for $I_D^{\mathrm{AYM}}$; the same argument would not work for arbitrary functionals.

\section{Relationship with J-stability, K-stability}

\subsection{Definitions}
With notations of the section \ref{chow-stabl-lins}, we are ready to introduce the notion of (uniform) J-stability. The notion of J-stability, which was modeled on the notion of K-stability, was introduced by Lejmi-Sz\'ekelyhidi \cite{L-Sz}. It is expected that J-stability of $(M,L_1,L_2)$ is equivalent to the existence of a critical metric.

\begin{definition}\label{jstabilityone} Under the setting of previous section, we define the J-weight $J_{L_2}(\scX,\scL)$ of a test configuration $(\scX,\scL)$ to equal (up to a positive multiplicative constant) the leading order term of the asymptotic Chow weight of the linear system $|L_2|$ (cf.~Theorem \ref{numericalstability}, Definition \ref{defasymChow}). Explicitly, the J-weight of a test configuration is $$J_{L_2}(\scX,\scL) = \frac{\hat{b}_0a_0 - b_0\hat{a}_0}{a_0}.$$ 
We define the \textit{ minimum norm} of a test configuration $(\scX,\scL)$  to be 
 $$\Vert (\scX,\scL)\Vert_m= J_{L}(\scX,\scL).$$
We say that 
\begin{itemize}
 \item $(M,L_1,L_2)$ is \emph{J-stable} (resp. J-semistable) if $$J_{L_2}(\scX,\scL)>0 \hspace{1cm}(\text{resp.}\,\,\geq0)$$  
 \item $(M,L_1,L_2)$ is \emph{uniformly J-stable} if for a uniform constant $\epsilon>0$  $$J_{L_2}(\scX,\scL)\geq \epsilon \Vert (\scX,\scL)\Vert_m \hspace{1.2cm}$$
\end{itemize}
for all non-trivial test configurations $(\scX,\scL)$ with normal total space.
\end{definition}

It is not difficult to see that  asymptotically Chow semistability in the sense of Definition \ref{defasymChow} implies J-semistability, see \cite{DK}. It remains unclear if it is true that J-stability implies Chow stability (cf.~\cite{RT2}).

\begin{remark}
 In the literature, several definitions of minimum norms have appeared that turn out to be equivalent. In \cite{Dervan2014}, it is justified that $\Vert.\Vert_m$ is a norm on the space of test configurations, vanishing only for almost trivial ones. See also \cite{BHJ,Sz2015}. \\
 Also, note that from the work of Lejmi-Sz\'ekelyhidi, we know that the existence of a critical metric implies the uniform J-stability. It is tempting to expect that the converse holds, given the result of Berman-Boucksom-Jonsson \cite{BBJ} for Fano manifolds.
\end{remark}

For comparison we state the definition of K-stability.

\begin{definition}\label{kstability} Let $(X,L)$ be a polarized normal variety. We define the Donaldson-Futaki invariant $\DF(\scX,\scL)$ of a test configuration to be (a positive constant times) the leading order term in its classical asymptotic Chow weight. Explicitly,
$$\DF(\scX,\scL) = \lim_{l \to \infty} \frac{e_{n+1}(lr)}{h(lr)} =  \frac{b_0a_1 - b_1a_0}{a_0}.$$ 
We say that 
\begin{itemize}
 \item $(X,L)$ is K-stable (resp. K-semistable) if $$\DF(\scX,\scL)>0  \hspace{1cm}(\text{resp.}\,\,\geq0)$$
 \item $(X,L)$ is uniformly K-stable if for a uniform constant $\epsilon>0$  $$\DF(\scX,\scL)\geq \epsilon \Vert (\scX,\scL)\Vert_m \hspace{1.2cm}$$
\end{itemize}
for all non-trivial test configurations with normal total space.\end{definition}

\subsection{Conditions for being uniformly J-stable or K-stable}\label{jflowsection}

Let us recall that a flag ideal $\scI$ is a coherent ideal sheaf on $M\times\comp$ of the form $\scI = I_0+(t)I_1+\hdots +(t^N)$, where $t$ is the coordinate on $\comp$ and $I_0\subset\hdots\subset I_{N-1}\subset\scO_M$ are a sequence of coherent ideal sheaves on $M$ corresponding to subschemes $Z_0\supset Z_1 \supset \hdots \supset Z_{N-1}$ of $M$. 
Denote by $\tilde{\scB}$ the blow-up of $\scI$ on $X\times\comp$, i.e. $$\pi: \tilde{\scB} = \Bl_{\scI} M\times\comp\to M\times\comp,$$ denote by $\scL_1,\scL_2$ the pullbacks of $L_1,L_1$ to $\tilde{\scB}$ and let $\scO(-E)=\pi^{-1}\scI$ be the exceptional divisor of the blow-up. The the natural $\comp^*$-action on $X\times\comp$ lifts to $\tilde{\scB}$ and 
it follows that $(\tilde{\scB},r\scL_1-E)$ is a test configuration of exponent $r$ for $(X,L_1)$ provided $r\scL_1-E$ is relatively ample. Remark that each such test configuration has a canonical compactification obtained by blowing up $\scI$ on $M\times\pr^1$; we denote this test configuration by $\scB$ and by abuse of notation denote $\scL_1,\scL_2,E$ the corresponding line bundles and divisors on $\scB$. 

In \cite{DK}, it is proved that to check J-stability of $(M,L_1,L_2)$, it is enough to consider the test-configurations that arise as blowups of flag ideals. We refer to 
\cite{DK} for details about this notion. The following proposition allows us to check J-stability by computing the J-weight using an intersection formula on $\tilde{\scB}$.

\begin{proposition}[\cite{DK}]\label{jflowblow-ups} Let $M$ be a normal projective variety with ample line bundles $L_1,L_2$. Then $(M,L_1,L_2)$ is J-stable if and only if $$J_{L_2}(\scB,r\scL_1-E)>0$$ for all flag ideals $\scI$ with $\scB=\Bl_{\scI}M\times\pr^1$ normal and $r\scL-E$ relatively semi-ample over $\pr^1$. Recall the J-constant of $(M,L_1,L_2)$ is defined as  $$\gamma = \frac{L_2.L_1^{n-1}}{L_1^n}.$$ Then the J-weight is on such blow-ups is given as \begin{equation}\label{numericaljweight}J_{L_2}(\scB,r\scL_1-E) = (r\scL_1-E)^n.\left(-\frac{n}{n+1}\gamma r^{-1}(r\scL_1-E) + \scL_2\right),\end{equation} up to multiplication by a positive dimensional constant.\end{proposition}

Using Equation \eqref{numericaljweight}, we can extend the definition of J-stability to the case where $L_2$ is an arbitrary line bundle. We now extend some results of \cite{DK} about J-stability. 

\begin{newTheorem}\label{sufficientconditionjstability} Consider a polarized manifold $(M,L_1)$ of complex dimension $n$ and let $L_2$ be a line bundle over $M$. Assume that $\gamma=\frac{L_2L_1^{n-1}}{L_1^n}>0$ and $\gamma L_1 - L_2$ is nef, then $(M,L_1,L_2)$ is uniformly J-stable.
\end{newTheorem}

\begin{proof} 
Let us treat the case of dimension $n>2$ first. As pointed out in \cite[Remark 3.11]{Dervan2014}, one can express the minimum norm of $(\mathcal{B},r\mathcal{L}_1-E)$ as
\begin{equation}\label{minnom}\Vert (\mathcal{B},r\mathcal{L}_1-E)\Vert_m = (r\mathcal{L}_1-E)^n.\left(\frac{1}{n+1}(r\mathcal{L}_1+nE)\right).\end{equation}
Now, from Equation (\ref{numericaljweight}), we have
\begin{align*}
 J_{L_2}(\scB,r\scL_1-E) &= (r\scL_1-E)^n.\left(-\frac{n}{n+1}\gamma r^{-1}(r\scL_1-E) + \scL_2\right),\\
 &=(r\scL_1-E)^n.\left(\frac{1}{n+1}\gamma r^{-1}(r\scL_1+nE) + (-\gamma\scL_1+\scL_2)\right),\\
 &=\frac{\gamma}{r}\Vert (\mathcal{B},r\mathcal{L}_1-E)\Vert_m + (r\scL_1-E)^n.(-\gamma\scL_1+\scL_2).
\end{align*}
The second term of the RHS is non-negative. Actually, from \cite[Lemma 3.7]{Derv} (see also \cite{Od-Sa}), one has for any nef divisor (e.g.~$\gamma L_1 - L_2$) on $M$ and $\mathcal{R}$ the induced divisor on $\mathcal{B}$ from the blow-up map (e.g.~$\gamma \scL_1 - \scL_2$), that $(r\scL_1-E)^n.\mathcal{R}\leq 0$. Consequently, applying this result with $\mathcal{R}=\gamma \scL_1 - \scL_2$ and  with our assumption, we obtain
$$J_{L_2}(\scB,r\scL_1-E)\geq \frac{\gamma}{r}\Vert (\mathcal{B},r\mathcal{L}_1-E)\Vert_m,$$
and thus uniform J-stability.
\end{proof}

We turn now to J-semistability, which can be regarded as the ``boundary case'' of the uniform J-stability. The first part of the following result shows that we obtain J-semistability by including the ``boundary'' in the conditions of Theorem \ref{sufficientconditionjstability}.

\begin{newTheorem}\label{Jsemi}
Consider a polarized manifold $(M,L_1)$ and let $L_2$ be a line bundle over $M$. Assume that $\gamma=\frac{L_2L_1^{n-1}}{L_1^n}\geq0$. 
 \begin{itemize}
 \item If $M$ has dimension $n>2$ and $\gamma L_1 - L_2$ is nef, or
 \item If $M$ has dimension $n=2$ and $\frac{4}{3}\gamma L_1 - L_2$ is nef, 
\end{itemize}
then $(M,L_1,L_2)$ is J-semistable.
\end{newTheorem}
\begin{proof}
 The case $n>2$ is a direct consequence of the previous theorem. For $n=2$, we simply notice that
 $$J_{L_2}(\scB,r\scL_1-E)=(r\scL_1-E)^2\left(\frac{2}{3}\frac{\gamma}{r}(r\scL_1+E)\right) +(r\scL_1-E)^2\left(-\frac{4}{3}\gamma \scL_1+\scL_2)\right).$$
As in the proof of Theorem \ref{sufficientconditionjstability}, the second term of the RHS is non-negative from our assumption. The first term is non-negative
from  \cite[Lemma 4.40]{DK}. 
\end{proof}
\begin{remark}\label{trivial}
 Note that when $\gamma=0$, $-L_2$ is nef. Now, from \cite[Theorem 1]{Luo90}, Hodge Index Theorem provides that either $L_2$ is numerically equivalent to 0 or $L_1^{n-2}L_2^2<0$. But the inequality cannot hold as $-L_2$ is nef.
\end{remark}

We can also slightly strengthen \cite[Theorem 4.36]{DK}.  This is the algebro-geometric analogue of the relationship between lower boundedness of the functional $\hat{J}$ (see \cite{C2}) and coercivity of the Mabuchi functional, due to Chen.  
\begin{newTheorem}
 \label{jimpliesk} Suppose $(M,L_1,K_M)$ is J-semistable, with $M$ having at worst Kawamata log terminal (klt) singularities. Then $(M,L_1)$ is uniformly K-stable. 
\end{newTheorem}
\begin{proof}
 To check K-stability it is sufficient to show the Donaldson-Futaki invariant for each test configuration given in Proposition \ref{jflowblow-ups} is strictly positive, see \cite[Corollary 3.11]{Odaka2}. We have $$J_{K_M}(\scB,r\scL_1-E) = (r\scL_1-E)^n.\left(-\frac{n}{n+1}\gamma r^{-1}(r\scL_1-E) + \scK_M\right),$$ while the corresponding Donaldson-Futaki invariant is given by Odaka as \begin{equation}\label{DFut} \DF(\scB,r\scL_1-E)\hspace{-0.1cm}=\hspace{-0.1cm}(r\scL_1-E)^{n}.\hspace{-0.05cm}\left(-\frac{n}{n+1}\gamma r^{-1}(r\scL_1-E)\hspace{-0.05cm}+\hspace{-0.05cm}\scK_M\hspace{-0.05cm}+\hspace{-0.05cm} K_{\scB/M\times\pr^1}\hspace{-0.05cm}\right).\end{equation} Here we have denoted $\scK_M$ the pullback of $K_M$ to $\scB$. \\
 Hence $$ \DF(\scB,r\scL_1-E) = J_{K_M}(\scB,\scL_1^r-E) + (r\scL_1-E)^n.K_{\scB/M\times\pr^1}.$$ 
 From \cite[Lemma 3.17]{Dervan2014}, $(r\scL_1-E)^n.K_{\scB/M\times\pr^1}>\epsilon \Vert (B,r\scL_1-E)\Vert_m$ where $\epsilon$ is a uniform constant; note that we need $M$ to have at worst klt singularities in order to apply this result. The conclusion follows.
\end{proof}

From Theorems \ref{Jsemi} and \ref{jimpliesk}, we obtain the next corollary. 
\begin{newcorollary}\label{kstablecone}
Assume that $\gamma=\frac{K_ML_2^{n-1}}{L_1^n}\geq 0$. 
\begin{itemize}
 \item Suppose that $(M,L_1)$ is a polarized Kawamata log terminal variety of dimension $n>2$ satisfying $\gamma L_1 - K_M$ is nef. Then $(M,L_1)$ is uniformly K-stable. 
\item Suppose that $(M,L_1)$ is a polarized Kawamata log terminal surface satisfying $\frac{4}{3}\gamma L_1 - K_M$ is nef. Then $(M,L_1)$ is uniformly K-stable. 
\end{itemize}
\end{newcorollary}
When $\gamma=0$, Remark \ref{trivial} ensures that actually that $M$ is a Calabi-Yau manifold. In that case K-stability was proved by Odaka \cite[Theorem 2.10]{Odaka3}. \\
The second part of the corollary has to be compared with recent results for surfaces where slightly different conditions have appeared in order to obtain properness of the Mabuchi K-energy, see \cite{FLSW,SW2}. See also Corollary \ref{coercivityMab} for the K\"ahler analogue of the above result on the coercivity of the Mabuchi K-energy.

\section{The K\"ahler non integral case}\label{nonintegralsect}
One can wonder if certain results of the previous section hold in the K\"ahler non integral case. We provide K\"ahler analogues of the results that we proved in Section \ref{jflowsection}. Although the results we provide below are parallel to the ones in Section \ref{jflowsection}, it is important to note that the proofs are completely different since algebro-geometric tools (such as intersection numbers of divisors) that we used in the projective case are not available in this setting.

Under the settings of Section \ref{intro}, we can introduce the functional on the space of $\omega$-potentials, 
$$\hat{J}_{\omega,\chi}(\phi)=\hat{J}(\phi)=\int_0^1 \int_M \dot{\phi_t}\left( \chi\wedge \omega_{\phi_t}^{n-1}-\gamma \omega_{\phi_t}^n\right)dt,$$ for $\omega_{\phi_t}$ a smooth K\"ahler path from $\omega$ to  $\omega_\phi:=\omega+\ddbar\phi$. The functional $\hat{J}$ is well defined and independent of the chosen path as integral of the moment map in the infinite
dimensional setup (see \cite{DK} where it is denoted $I_{\mu_J}$). We notice that its definition can be extended to the case of $\chi$ is a close $(1,1)$-form not necessarily definite positive.\\
We discuss now the notion of coercivity, and refer for details to \cite[Chapter 3]{Siu} and \cite[Section 6.1]{Ti2}. Let us recall that we say that a functional $\mathcal{F}$ on the space of $\omega$-potentials $\phi$ is coercive if $$\mathcal{F}(\phi)\geq \alpha_1 \mathcal{I}_\omega(\phi)+\alpha_2$$
for uniform constants $\alpha_1,\alpha_2$ with $\alpha_1>0$ where $\mathcal{I}_\omega(\phi)=\int_M \phi(\omega^n-\omega_\phi^n)\geq 0$ is a classical energy functional. 
In the literature, it is also generally introduced the functional $$J^{AYM}_\omega(\phi) =\int_0^1\int_M\dot{\phi_t}(\omega^n-\omega_{\phi_t}^n)$$ for $\phi_t$ a path of $\omega$-potentials between 0 and $\phi$. The functionals $\mathcal{I}_\omega,J^{AYM}_\omega, \mathcal{I}_\omega-J^{AYM}_\omega$ are equivalent, as $$(n+1)J^{AYM}_\omega(\phi)\geq \mathcal{I}_\omega(\phi)\geq \frac{n+1}{n}J^{AYM}_\omega(\phi).$$
Thus the definition of coercivity can be made with respect to any of these three functionals. It is also well-known that one can write
\begin{equation}\label{Jfunc}J^{AYM}_\omega(\phi)=\int_M \phi \omega^n - \frac{1}{n+1}\sum_{i=0}^n \int_M \phi \omega^i \wedge (\omega+\ddbar\phi)^{n-i}.\end{equation}
Of course, the coercivity implies the \textit{properness property} (i.e if $\mathcal{I}_\omega(\phi)\to +\infty$ then $\mathcal{F}(\phi)\to +\infty$).\\
The following theorem is a consequence of the techniques of \cite{Derv3}.
\begin{newTheorem}\label{propernessK}
Assume that $M$ is a compact K\"ahler manifold, $\omega$ is a K\"ahler form and $\chi$ a closed $(1,1)$-form. Assume that  $\gamma=\frac{[\chi][\omega]^{n-1}}{[\omega]^n}>0$ and $\gamma [\omega] - [\chi]$ is nef, then the functional $\hat{J}_{\omega,\chi}$ is coercive on the space of $\omega$-potentials. If $\gamma=0$ then the functional $\hat{J}_{\omega,\chi}$ is bounded from below.
\end{newTheorem}
\begin{remark}
	For the definition of nefness for (non integral) K\"ahler classes, the reader is referred e.g.~to \cite[Chapter 17]{Dembk}.
\end{remark}

\begin{proof}
We denote $\Vol$ the volume of the class $[\omega]$ over $M$. We know from Chen \cite{C2} that $\hat{J}$ writes as
 \begin{align*}
  \hat{J}(\phi)=-\frac{n}{n+1}\gamma \int_M \phi\sum_{i=0}^n \omega^i \wedge \omega_\phi^{n-i} + \int_M \phi \chi \wedge \sum_{i=0}^{n-1} \omega^i \wedge \omega_\phi^{n-1-i}
  \end{align*}
and without loss of generality,  restrict to the case $\sup_M \phi=0$.
 From our assumption, $\gamma \omega-\chi$ can be written as  $\xi+ \sqrt{-1}\partial\bar\partial\psi$ with possibly singular $\xi\geq 0$ and $\psi$ smooth.
 Now, 
 \begin{align*}
  \hat{J}(\phi)=& -\frac{n}{n+1}\gamma\int_M \phi\sum_{i=0}^n \omega^i \wedge \omega_\phi^{n-i} \\
&  + \int_M \phi (\chi -\gamma\omega ) \sum_{i=0}^{n-1} \omega^i \wedge
\omega_\phi^{n-1-i}  +\gamma\int_M \phi \omega\wedge\sum_{i=0}^{n-1} \omega^i \wedge \omega_\phi^{n-1-i},\\
=&    -\frac{n}{n+1}\gamma\int_M \phi\sum_{i=0}^n \omega^i \wedge \omega_\phi^{n-i} \\
&  - \int_M \phi (\xi+ \sqrt{-1}\partial\bar\partial\psi) \sum_{i=0}^{n-1} \omega^i \wedge
\omega_\phi^{n-1-i} +\gamma\int_M \phi \omega\wedge\sum_{i=0}^{n-1} \omega^i \wedge \omega_\phi^{n-1-i}.
 \end{align*}
Let's see that the second term $ - \int_M \phi (\xi+ \sqrt{-1}\partial\bar\partial\psi) \sum_{i=0}^{n-1} \omega^i \wedge
\omega_\phi^{n-1-i}$ is bounded from below, by decomposing this expression as follows. As $\xi$ is non-negative and $\sup_M\phi=0$, we have
\begin{align*}
 \int_M \phi\xi\wedge \sum_{i=0}^{n-1} \omega^i \wedge \omega_\phi^{n-1-i} & \leq \sup_M\phi \int_M \xi\wedge \sum_{i=0}^{n-1} \omega^i \wedge \omega_\phi^{n-1-i},\\
    &\leq 0.
\end{align*}
Moreover, 
\begin{align*}
\int_M  \sqrt{-1}\phi\partial\bar\partial\psi \sum_{i=0}^{n-1} \omega^i \wedge \omega_\phi^{n-1-i}&=\int_M \psi(\omega_\phi-\omega)\sum_{i=0}^{n-1} \omega^i \wedge \omega_\phi^{n-1-i},\\
 &=\int_M \psi (\omega_\phi^n -\omega^n),\\
 &\leq \int_M (\sup_M \psi)\omega_\phi^n- \int_M (\inf_M \psi)\omega^n,\\
 &\leq (\sup_M \psi -\inf_M\psi)\Vol.
\end{align*}
Hence  $ - \int_M \phi (\xi+ \sqrt{-1}\partial\bar\partial\psi) \sum_{i=0}^{n-1} \omega^i \wedge
\omega_\phi^{n-1-i}$ is bounded from below by a constant $c_1$ independent of $\phi$.
Then, using \eqref{Jfunc},
\begin{align*}
  \hat{J}(\phi)\geq &-\frac{n}{n+1}\gamma\int_M \phi\sum_{i=0}^n \omega^i \wedge \omega_\phi^{n-i} +c_1+\gamma\int_M \phi \sum_{i=1}^n \omega^i \wedge \omega_\phi^{n-i},\\
 \geq &\gamma\left( \frac{1}{n+1}\int_M \phi\sum_{i=0}^n \omega^i \wedge \omega_\phi^{n-i} -\int_M\phi \omega_\phi^n \right)+c_1,\\
  \geq &\gamma\left( \frac{1}{n+1}\int_M \phi\sum_{i=0}^n \omega^i \wedge \omega_\phi^{n-i} -\int_M \phi \omega^n + \int_M\phi (\omega^n-\omega_\phi^n) \right)+c_1,\\
  \geq &\gamma\left(\mathcal{I}_\omega(\phi)-J_\omega^{AYM}(\phi)\right)+c_1.
 \end{align*}.
\end{proof}

The above theorem leads to the following result, which can be regarded as a K\"ahler analogue of the first part of Corollary \ref{kstablecone}.
\begin{newcorollary}\label{coercivityMab}
Assume that $M$ is a compact K\"ahler manifold, $\omega$ is a K\"ahler form. Assume that  $\gamma=\frac{[K_M][\omega]^{n-1}}{[\omega]^n}\geq 0$ and $\gamma [\omega] - [K_M]$ is nef, then the Mabuchi functional for $[\omega]$ is coercive. 
\end{newcorollary}
\begin{proof}
 We recall that Mabuchi K-energy is given by
 \begin{align*}K_\omega(\phi)&=\int_M \log\frac{\omega_\phi^n}{\omega^n} \omega^n_\phi + \hat{J}(\phi),
 \end{align*}
 and that the entropy term $H(\phi)=\int_M \log\frac{\omega_\phi^n}{\omega^n}\omega^n_\phi $ is a coercive functional, see for instance \cite[Theorem 7.13]{Ti2}. We conclude by applying the previous theorem.
\end{proof}
\begin{remark}
	Let us mention that for Fano manifolds a similar result with ample classes appears in \cite{Li-Shi-Yao} and \cite{Derv3}. See also \cite{Zheng2015}.
\end{remark}

In \cite{S2016} and in \cite{DR2016}, it is defined a notion of (uniform) J-stability and K-stability \textit{in the K\"ahler sense}, i.e for a possibly non integral K\"ahler class $[\omega]$ over a compact K\"ahler manifold (see also related results in \cite{S2016}).  In that case the J-weight $J_{[\chi]}$, the Donaldson-Futaki invariant, and the minimum norm are given as intersection numbers on a resolution of singularities of the test configuration by an expression of the form \eqref{numericaljweight}, \eqref{DFut} and \eqref{minnom}. We refer to the paper of Dervan and Ross for details. The main difference with the projective case, is that a test configuration is defined as a normal K\"ahler space $(\mathcal{X},\Omega)$, where the $\mathbb{C}^*$ action on $\mathcal{X}$ covers the one on $\mathbb{P}^1$ and $\Omega$ is a  $S^1$-invariant K\"ahler metric. The fibers $(\mathcal{X}_t,\Omega_t)$ are isomorphic to $(M,\omega)$ for $t\neq 0$. With these notions in hands, the definitions of (uniform) J-stability and K-stability in the K\"ahler setting are identical to Definitions \ref{jstabilityone} and \ref{kstability}.\\
If one fixes a test configuration $(\mathcal{X},\Omega)$ for $(M,[\omega])$, Dervan and Ross related the J-weight of $(\mathcal{X},\Omega)$ to the derivative of the $\hat{J}_{\omega,\chi}$ functional. More precisely, they  proved the following theorem.
\begin{theorem}\cite[Theorem 6.4]{DR2016} The $\mathbb{C}^*$ action coming from the test configuration   $(\mathcal{X},\Omega)$ provides a biholomorphism $\rho(t)$ between the fibers $\mathcal{X}_1$ and $\mathcal{X}_t$ for $t\neq 0$.
Define $\rho(\tau)^*\Omega_\tau=\Omega_1+\ddbar \theta_\tau$, $t=-\log\vert \tau\vert^2$ and $\phi_t=\theta_\tau$. Then,
$$J_{[\chi]}(\mathcal{X},\Omega) = \lim_{t \to \infty} \frac{d}{dt}\hat{J}_{\omega,\chi}(\phi_t).$$
\end{theorem}
In \cite{DR2016}, the definition of minimum norm of the test configuration $ (\mathcal{X},\Omega)$ is different from the one we are working with, but both norms turn out to be equivalent. The following relationship \eqref{min} is pointed out in  \cite[Remark 2.8 (i)]{DR2016} and obtained in a similar way to \cite[Theorem 4.16]{DR2016},
\begin{equation}
\label{min}\Vert (\mathcal{X},\Omega)\Vert_m = \lim_{t \to \infty} \frac{d}{dt}(\mathcal{I}_\omega - J^{AYM}_\omega)(\phi_t).
\end{equation}

This leads to the following result, which can be compared to Theorem \ref{sufficientconditionjstability} for smooth projective varieties.
\begin{newcorollary}\label{kahlerJ}
 Assume that $M$ is a compact K\"ahler manifold, $\omega$ is a K\"ahler form and $\chi$ a closed $(1,1)$-form. Assume that  $\gamma=\frac{[\chi][\omega]^{n-1}}{[\omega]^n}>0$ and $\gamma [\omega] - [\chi]$ is nef, then $(M,\omega)$ is uniformly J-stable (with respect to $\chi$) in the K\"ahler sense. If $\gamma=0$ then $(M,\omega)$ is J-semistable (with respect to $\chi$) in the K\"ahler sense.
\end{newcorollary}
\begin{proof}
  We apply Theorem \ref{propernessK}. At infinity, for any $1>\varepsilon>0$,
 $$\hat{J}(\phi_t) - (1-\varepsilon)\gamma(\mathcal{I}_\omega - J^{AYM}_\omega)(\phi_t)$$ is increasing as $(\mathcal{I}_\omega - J^{AYM}_\omega)(\phi_t)\to +\infty$. 
 Consequently, we have $$\frac{d}{dt}\hat{J}_{\omega,\chi}(\phi_t)\geq (1-\varepsilon)\gamma\frac{d}{dt}(\mathcal{I}_\omega - J^{AYM}_\omega)(\phi_t).$$
Taking the limit gives that  for any $1>\varepsilon>0$,
$$J_{[\chi]}(\mathcal{X},\Omega)\geq (1-\varepsilon)\gamma\Vert ({\chi},\Omega)\Vert_m$$
and thus uniform J-stability for the K\"ahler class $[\omega]$. The case $\gamma=0$ is treated in a similar way. 
\end{proof}

The following result can be regarded as a K\"ahler analogue of Theorem \ref{Jsemi}.
\begin{newcorollary}
Assume that $M$ is a compact K\"ahler manifold, $\omega$ is a K\"ahler form. Assume that  $\gamma=\frac{[K_M][\omega]^{n-1}}{[\omega]^n}\geq 0$ and $\gamma [\omega] - [K_M]$ is nef, then $(M,[\omega])$ is uniformly K-stable in the K\"ahler sense. 
\end{newcorollary}
\begin{proof}
 We apply \cite[Theorem 4.14]{DR2016}. The Donaldson-Futaki invariant is given by the limit of the derivative of the Mabuchi K-energy along the 1-parameter family of biholomorphisms $\rho(t)$. Now, from Corollary \ref{coercivityMab}, the derivative of the Mabuchi K-energy is bounded from below by the derivative
 of the $\mathcal{I}_\omega - J^{AYM}_\omega$ functional (up to a multiplicative constant) and we can take the limit.
\end{proof}

As a direct consequence of this last theorem, we see that any K\"ahler class of a (non projective) Calabi-Yau is uniformly K-stable. This last fact can be proved using  the technology of Dervan-Ross and Berman-Darvas-Lu, see \cite[Corollary 4.17]{DR2016}, but the point is that our approach does not involve any deep analytic result.

\bibliography{hdr.bib}

\bigskip 

\address{{\small
\noindent\begin{tabular}{rp{7.8cm}}
 {\sc Yoshinori Hashimoto} & {\sc Data4C's K.K., 5-2-32, Minamiazabu, Minato-ku, Tokyo, Japan.}\\
 & {\sc Aix Marseille Universit\'e, CNRS, Centrale Marseille, Institut de Math\'ematiques
de Marseille, UMR 7373, 13453 Marseille, France.}\\
{\it Email address: } &\email{hashimoto\_yoshinori@data4cs.co.jp}\\
                      &\email{yoshinori.hashimoto@univ-amu.fr}\\
                      &  \\
{\sc Julien Keller} &  {\sc Aix Marseille Universit\'e, CNRS, Centrale Marseille, Institut de Math\'ematiques
de Marseille, UMR 7373, 13453 Marseille, France.}\\
{\it Email address: }& \email{julien.keller@univ-amu.fr}\\
\end{tabular}}}

\end{document}